\documentclass{article}
\usepackage{amsthm,amsfonts,amsmath,amssymb}

\newcommand{\identity}{\iota}

\newcommand{\ceil}[1]{\lceil #1 \rceil}

\newcommand{\Erdos}{Erd\H{o}s}

\newcommand{\ignore}[1]{}

\newtheorem{theorem}{Theorem}
\newtheorem{lemma}[theorem]{Lemma}
\newtheorem{defn}{Definition}
\newtheorem{proposition}[theorem]{Proposition}

\title{Longest Common Subsequences in Sets of Permutations}
\author{Paul Beame\thanks{Research supported by NSF grants CCF-0514870 and CCF-0830626}\\
              \small Computer Science and Engineering \\[-0.8ex]
              \small University of Washington \\[-0.8ex]
              \small Seattle, WA \\[-0.8ex]
              \small \texttt{beame@cs.washington.edu}
     \and Eric Blais\thanks{Research supported by a scholarship from the Fonds qu\'eb\'ecois de la recherche sur la nature et les technologies (FQRNT).}  \\
              \small School of Computer Science\\[-0.8ex]
              \small Carnegie Mellon University\\[-0.8ex]
              \small Pittsburgh, PA \\[-0.8ex]
              \small \texttt{eblais@cs.cmu.edu}
     \and Dang-Trinh Huynh-Ngoc\thanks{Research supported by NSF grant CCF-0830626 and a Vietnam Education Foundation Fellowship} \\
              \small Computer Science and Engineering \\[-0.8ex]
              \small University of Washington \\[-0.8ex]
              \small Seattle, WA \\[-0.8ex]
              \small \texttt{trinh@cs.washington.edu}
     }
\date{\today}

\begin{document}

\maketitle

\begin{abstract}
The sequence $a_1\ \cdots\ a_m$ is a \emph{common subsequence} in the set of
permutations
$S = \{\pi_1,\ldots,\pi_k\}$ on $[n]$ if it is a subsequence of
$\pi_i(1)\ \cdots\ \pi_i(n)$ and $\pi_j(1)\ \cdots\ \pi_j(n)$ for some distinct
$\pi_i, \pi_j \in S$.  Recently, Beame and Huynh-Ngoc (2008) showed that when
$k \ge 3$, every set of $k$ permutations on $[n]$ has a common subsequence of
length at least $n^{1/3}$.  

We show that, surprisingly, this lower bound is asymptotically optimal for all constant values of $k$. Specifically, we show that
for any $k \ge 3$ and $n \ge k^2$ there exists a set of $k$ permutations on $[n]$ in which
the longest common subsequence has length at most $32(kn)^{1/3}$.
The proof of the upper bound is constructive, and uses elementary algebraic
techniques. 
\end{abstract}

\newpage
\section{Introduction}

The sequence $a_1\ \cdots\ a_m$ is a \emph{common subsequence} in the set  $S = \{\pi_1,\ldots,\pi_k\}$ of
permutations on $[n]$ if it is a subsequence of
$\pi_i(1)\ \cdots\ \pi_i(n)$ and $\pi_j(1)\ \cdots\ \pi_j(n)$ for some distinct
$\pi_i, \pi_j \in S$.  
In this article, we study the minimum length of the longest common subsequence(s) in a set of $k$ permutations on $n$.  

\begin{defn}
Let $f_k(n)$ denote the maximum value $m$ for which every set of $k$ permutations on $[n]$ is guaranteed to contain a common subsequence of length $m$.
\end{defn}

The celebrated \Erdos-Szekeres Theorem~\cite{ES35} states that every sequence of length $n$ contains a monotone subsequence of length $\ceil{n^{1/2}}$.  In our terminology, the theorem states that for every permutation $\pi$ on $n$, the set $\{\pi, \iota, \iota^R\}$ contains a common subsequence of length at least $\ceil{n^{1/2}}$, where $\iota$ is the identity permutation and $\iota^R$ is its reversal.  

As a consequence of  the \Erdos-Szekeres Theorem, sets of permutations that include a permutation and its reversal can not hope to show an upper bound stronger than $f_3(n) \le \ceil{n^{1/2}}$.  The bound on $f_3(n)$ is in fact much smaller: as Beame and Huynh-Ngoc~\cite{BH08} recently showed, $f_3(n) = f_4(n) = \ceil{n^{1/3}}$. 

When $k > 4$, the exact values of the function $f_k(n)$ are not currently known.  A simple probabilistic argument establishes an upper bound of 
\begin{equation}
\label{eqn:upperbound}
f_k(n) < 2e\sqrt{n}
\end{equation}
for every $k < e^{e\sqrt{n}}$, and a counting argument shows that for every $k \ge 3$, 
\begin{equation}
\label{eqn:lowerbound}
f_k(n) \ge \ceil{n^{1/3}}.
\end{equation} (Proofs of (\ref{eqn:upperbound}) and (\ref{eqn:lowerbound}) are included in the Appendix.)  The goal of the research presented in this note was to determine the correct asymptotic behavior of $f_k(n)$.

\subsection{Our results}

We present two results in this paper.  The first result uses Hadamard matrices to show that $f_k(n)$ grows asymptotically slower than $\sqrt{n}$.

\begin{theorem}
\label{thm:hadamard}
Let $k \ge 4$ be an integer such that a Hadamard matrix of order $k$ exists.  Then
$
f_k(n) \le \ceil{n^{\frac{1}{k-1}}}^{k/2 - 1}.
$
\end{theorem}

A slightly weaker form of Theorem~\ref{thm:hadamard} 
was mentioned in a preprint~\cite{BH08-eccc} of~\cite{BH08} but only
the details for the case $k=4$ were included.
Beame and Huynh-Ngoc also
conjectured in~\cite{BH08-eccc} that the bound in Theorem~\ref{thm:hadamard} is
tight, up to a multiplicative constant, for every $k$ power of 2.  

Our main result disproves the conjecture, showing that $f_k(n)$ grows at a rate proportional to
$n^{1/3}$ for every constant $k$.

\begin{theorem}
\label{thm:main}
For $3 \le k \le n^{1/2}$, 
$
f_k(n) \le 32(k n)^{1/3}.
$
\end{theorem}

Combined with the lower bound in  (\ref{eqn:lowerbound}), Theorem~\ref{thm:main} completely characterizes the behavior of $f_k(n)$ for every constant $k$, up to a multiplicative constant.


\subsection{Motivation and other related work}
\paragraph{Streaming algorithms.}
The behavior of $f_k(n)$ was first examined in~\cite{BH08} while studying the read/write streaming computation model introduced by Grohe and Schweikardt~\cite{GS05}. In this model, an algorithm can store an unlimited amount of temporary data in multiple auxiliary streams, but tries to minimize
both its memory size requirements and the total number of passes it makes on the data streams. 

In~\cite{BH08}, lower bounds on $f_k(n)$ were shown to give complexity upper
bounds on algorithms for the \emph{permuted promise set-disjointness} problem, an important problem in the read/write stream model.  In particular, the bound $f_3(n)\ge n^{1/3}$ was used to show the existence of an
algorithm that requires only logarithmic memory and a constant number of passes
when $n<p^{3/2}/64$, where $n$ is the size of the universal set and $p$ is the number of input subsets.   The conjecture in~\cite{BH08} that $f_k(n) \ge n^{1/2 - o(1)}$ for some $k$ that is $n^{o(1)}$ would have improved the result to show that the same algorithm would also work for any $n < p^{2 - o(1)}$, matching the lower bound for the problem.  Our result, however, strongly refutes the conjecture.

\paragraph{Error-correcting codes.}
A \emph{code} over a metric space $(M, d)$ is a set $C$ of elements -- called \emph{codewords} -- from $M$.  The code $C$ has distance $\delta$ if for every two distinct codewords $c_1, c_2 \in C$, $d(c_1,c_2) \ge \delta$.  Two central problems in the study of error-correcting codes involve determining the largest code with a given distance, and the dual problem of identifying the maximal distance of any code with $|C| = k$ codewords.  

In the study of codes that correct deletion errors, the metric of interest is the \emph{deletion distance}, where $d_{del}(\pi, \sigma)$ is one half the number of deletions and insertions required to turn the sequence $\pi(1)\ \cdots \ \pi(n)$ to $\sigma(1)\ \cdots \ \sigma(n)$.  Our results on $f_k(n)$ have a direct implication for codes built over $(\mathcal{S}_n, d_{del})$: a code of size $k$ over this metric space has maximal distance $n - f_k(n)$.

There has been extensive research on error-correcting codes built over a metric space defined by a deletion distance~\cite{AEL95,Lev66,SZ99,Slo00}, and on codes built over the symmetric group $\mathcal{S}_n$~\cite{BCD79,CCD04,DV78,FD}.  As far as we know, however, our result is the first to explicitly provide bounds on the capabilities of error-correcting codes built over $(\mathcal{S}_n, d_{del})$.

\paragraph{Combinatorics on sets of permutations.}
The study of $f_k(n)$ falls into the area of combinatorics on sets of permutations, an area that extends beyond the field of error-correcting codes.
In particular, we highlight the exciting recent result of Ellis, Friedgut, and Pilpel~\cite{EFP}, who settled a conjecture of Frankl and Deza~\cite{FD} by showing that for any set $S$ of permutations on $n$ in which the Hamming distance between every pair of distinct $\pi, \sigma \in S$ is $d_{Ham}(\pi, \sigma) \le n -m$, the size of $S$ must be at most $(n-m)!$.

\section{Proof of Theorem~\ref{thm:hadamard}}

Recall that a Hadamard matrix $H$ of order $k$
is a $k \times k$ $\pm 1$-matrix with the property that every two distinct rows in $H$ differ
in exactly $k/2$ entries. We use the rows of Hadamard matrices to construct $k$ permutations that
have no long common subsequence.

\newtheorem*{hadamardthm}{Theorem~\ref{thm:hadamard}}
\begin{hadamardthm}[Restated]
Let $k \ge 4$ be an integer such that a Hadamard matrix of order $k$ exists.  Then
$$
f_k(n) \le \big\lceil n^{\frac{1}{k-1}} \big\rceil^{k/2 - 1}.
$$
\end{hadamardthm}

\begin{proof}
Let $s=\big\lceil n^{\frac{1}{k-1}} \big\rceil$ and $n' = s^{k-1}\ge n$.
We will show that $f_k(n')\le s^{k/2-1}$.

There is a natural bijection $\phi$ from
$[n']$ into the $(k-1)$-dimensional integer lattice $[s]^{k-1}$
given by
$$
\phi(x) = \big( \phi_1(x), \ldots, \phi_{k-1}(x) \big)
$$
where $\phi_i(x)$ is the $i$-th digit of $x-1$ in base $s$,
with the left-most digit being the most significant. 
Note that under $\phi$, the standard ordering on $[n']$ induces the
lexicographic ordering on the vectors in $[s]^{k-1}$.

The idea of the construction is to use the $i$-th row of the Hadamard matrix to 
define $k-1$ permutations on $[s]$.   The $i$-th permutation in the set
is then chosen as the ``outer product'' of these $k-1$ permutations.  

More precisely, let $H$ be a $k\times k$ Hadamard matrix whose rows and
columns are indexed by $\{0,\ldots,k-1\}$ and whose first row and column entries
(without loss of generality) are all 1.
Define the permutation $\pi_{i,\ell} : [s] \to [s]$, depending on the last
$k-1$ columns of $H$, by
$$
\pi_{i,\ell} = \begin{cases}
\identity & \mbox{ if } H_{i,\ell} = 1 \\
\identity^R & \mbox{ if } H_{i,\ell} = -1,
\end{cases} 
$$
where $\identity$ is the identity permutation and
$\identity^R$ is the reversal permutation.
The permutation $\pi_i$ is then given by
$$
\pi_i(x) = \phi^{-1}\big( \pi_{i,1}( \phi_1(x)), \ldots, \pi_{i,k-1}(
\phi_{k-1}(x)) \big).
$$

Because $\phi^{-1}$ converts the lexicographic order on $[s]^{k-1}$ to the
standard order on $[n']$, the relative order of distinct elements $x,y \in [n']$
in $\pi_i$ depends only on their relative order in $\pi_{i,\ell}$ for
the first (most-significant) coordinate $\ell\in [k-1]$ such that
$\phi_\ell(x)\ne \phi_\ell(y)$; in particular, since the only choices for
$\pi_{i,\ell},\pi_{j,\ell}$ are $\identity$ and $\identity^R$, it follows that
$x$ and $y$ have the same relative order in $\pi_i$ and in $\pi_j$
if and only if $\pi_{i,\ell} = \pi_{j,\ell}$.

We reason by contradiction.
Assume that there exist two permutations $\pi_i, \pi_j$ in 
$\{\pi_1,\ldots,\pi_k\}$ that have a common
subsequence of length greater than $s^{k/2-1} = \big\lceil n^{\frac{1}{k-1}} \big\rceil^{k/2-1}$.
Let $L_{i,j}$ be the set of indices of columns among the last $k-1$ columns
of $H$
in which the rows $i$ and $j$ have the same value in $H$. 
By assumption on $H$, we have $|L_{i,j}| = k/2-1$.  So, by the Pigeonhole
Principle,
there must exist distinct $x,y$ in the common subsequence of $\pi_i$ and
$\pi_j$ such that $\phi_{\ell}(x) = \phi_{\ell}(y)$ for every
$\ell \in L_{i,j}$. 
But then $\pi_{i,\ell} \neq \pi_{j,\ell}$ for the first index $\ell$ such
that $\phi_\ell(x) \neq \phi_\ell(y)$, and so $x$ and $y$ do not have the same
relative order in $\pi_i$ and $\pi_j$.  This contradicts the fact that $x$
and $y$ are in a common subsequence of $\pi_i$ and $\pi_j$ and completes
the proof of the theorem.
\end{proof}

\section{Main result}

\newtheorem*{mainthm}{Theorem~\ref{thm:main}}
\begin{mainthm}[Restated]
For every $3 \le k \le n^{1/2}$, 
$$
f_k(n) \le 32(k n)^{1/3}.
$$
\end{mainthm}

There are two main ingredients used in the construction that establishes the
upper bound of $f_k(n)$ in  Theorem~\ref{thm:main}: a bijection $\phi_{n,k}$
that maps the integers $1,\ldots,n$ to a 3-dimensional integer lattice, and
$k$ triples of functions
$(g_{1,1}, g_{2,1}, g_{3,1}), \ldots, (g_{1,k}, g_{2,k}, g_{3,k})$ that are
used to generate orderings of the elements of the 3-dimensional lattice.

Let $s_1=(n/k^2)^{1/3}$ and $s_2=s_3=(nk)^{1/3}$.   
For simplicity, we first assume that $s_1,s_2,$ and $s_3$ are integers.
The general case will be easily dealt with later.
Since $k\le n^{1/2}$, we have $s_1\ge 1$ and $s_2=s_3\le n^{1/2}$.
Let $X$ be the 3-dimensional integer lattice $[s_1]\times [s_2]\times [s_3]$.
and let the bijection $\phi_{n,k} : [n] \to X$ be the function whose inverse is
 given by
$$
\phi_{n,k}^{-1}(x,y,z) = x + s_1 (y - 1) + s_1s_2(z - 1).
$$
This mapping associates the standard ordering on $[n]$ with the lexicographic
ordering on $(x,y,z)$ tuples in $X$ in which the $x$ coordinate is least
significant
and the $z$ coordinate is most significant.  The reason for the smaller range
of $x$ coordinates relative to the other two will become apparent in the
analysis.

Let $p$ be the smallest prime larger than $4s_3$.  For $j=1,\ldots,k$,
the functions $g_{1,j} : X \to \mathbb{Z}$, $g_{2,j} : X \to \mathbb{Z}$,
and $g_{3,j} : X \to \mathbb{Z}$ are defined by
\begin{eqnarray*}
g_{3,j}(x,y,z) &=& j^2x + 2jy + 2z \bmod{p}, \\
g_{2,j}(x,y,z) &=&  jx + y \mbox{, and} \\
g_{1,j}(x,y,z) &=& x.
\end{eqnarray*}

For every $i \in \{1,2,3\}$ and $j \in \{1,\ldots,k\}$, define
$h_{i,j} = g_{i,j} \circ \phi_{n,k}$ and set $h_j=(h_{1,j},h_{2,j},h_{3,j})$.  
Note that although the image of $[n]$ under $\phi_{n,k}$ is the set $X$, the
image of $[n]$ under an $h_j$ is a set of triples not constrained to
lie in $X$.  We first see that each $h_j$ is 1-1 on $[n]$.

\begin{proposition} \label{thm:order}
For any $j\in [k]$ and distinct $a,b \in [n]$ we have $h_j(a)\ne h_j(b)$.
\end{proposition}

\begin{proof}
Suppose for contradiction that there exist $a,b\in [n]$ such that
$h_{1,j}(a)=h_{1,j}(b)$, $h_{2,j}(a)=h_{2,j}(b)$, and $h_{3,j}(a)=h_{3,j}(b)$.
Then, by definition,  there exist two distinct points
$(x_a,y_a,z_a), (x_b,y_b,z_b)\in X$ such that
\begin{eqnarray*}
j^2x_a + 2jy_a + 2z_a &\equiv& j^2x_b + 2jy_b + 2z_b \pmod{p}, \\
jx_a+y_a &=&  jx_b + y_b, \qquad \mbox{and} \\
x_a &=& x_b,
\end{eqnarray*}
which implies that $x_a=x_b, y_a=y_b$, and $z_a\equiv z_b \pmod{p}$. Since
$p>s_3$, we have contradiction.
\end{proof}

For $j = 1,\ldots,k$, the function $h_j$ determines
a total order $<_j$ on $[n]$ as follows: 
For $a,b\in [n]$, write $a <_j b$ iff $h_j(a)$ is less than $h_j(b)$
in the lexicographic order on integer triples in which
the third coordinate is most significant and the first coordinate is the least
significant.

Let $\pi_j$ be the permutation on $[n]$ that orders the elements in $[n]$ in
increasing order as defined by $<_j$.  That is, let $\pi_j$ be the
permutation such that
$$
\pi_j(1) <_j \pi_j(2) <_j \cdots <_j \pi_j(n) .
$$
As we show below, the set of permutations $\{\pi_1,\ldots,\pi_k\}$ has no
common subsequence of length greater than $16(nk)^{1/3}$.

\begin{lemma}
\label{thm:lcs}
For $3 \le k \le n^{1/2}$ let $\pi_1,\ldots,\pi_k$ be the $k$
permutations on $\{1,\ldots,n\}$ defined above.  Then
$\{\pi_1, \ldots, \pi_k\}$ has no common subsequence of length greater than
$16(nk)^{1/3}$.
\end{lemma}

\begin{proof}
Let $a_1\ a_2\ \cdots\ a_s$ be a subsequence of $\pi_i$ and $\pi_j$, for some
$1 \le i < j \le k$.  Then
$$
\begin{array}{l}
a_1 <_i a_2 <_i \cdots <_i a_s \mbox{ and} \\
a_1 <_j a_2 <_j \cdots <_j a_s.
\end{array}
$$
In particular, this implies that
$h_{3,i}(a_1) \le h_{3,i}(a_2) \le \cdots \le h_{3,i}(a_s)$ and
$h_{3,j}(a_1) \le h_{3,j}(a_2) \le \cdots \le h_{3,j}(a_s)$. 
The functions $h_{3,i}$ and $h_{3,j}$ can each take $p$ different values, so
any sequence of distinct pairs
$(h_{3,i}(a_1), h_{3,j}(a_1)), \ldots, (h_{3,i}(a_s), h_{3,j}(a_s))$ satisfying the increasing property can have at most $s = 2p - 1$ elements. 
Bertrand's Postulate guarantees that $p < 8(nk)^{1/3}$, so to prove the claim
it is sufficient to show that the pairs
$\big( h_{3,i}(a_t), h_{3,j}(a_t) \big)$ for $t \in \{1,\ldots,s\}$ must be
distinct.

We prove by contradiction that the pairs
$\big( h_{3,i}(a_t), h_{3,j}(a_t) \big)$ must be distinct for $i\ne j$. 
Assume that there exist two indices $t \neq t'$ in $\{1,\ldots,s\}$ such that 
\begin{eqnarray*}
h_{3,i}(a_t) &=& h_{3,i}(a_{t'}) \qquad\text{and} \\
h_{3,j}(a_t) &=& h_{3,j}(a_{t'}) .
\end{eqnarray*}
Then, letting $\phi_{n,k}(a_t) = (x_t, y_t, z_t)$ and
$\phi_{n,k}(a_{t'}) = (x_{t'}, y_{t'}, z_{t'})$, the above equivalences
imply that
\begin{eqnarray*}
i^2(x_t - x_{t'}) + 2i(y_t - y_{t'}) + 2(z_t - z_{t'}) &\equiv& 0 \pmod{p} \text{\ \ and} \\
 j^2(x_t - x_{t'}) + 2j(y_t - y_{t'}) + 2(z_t - z_{t'}) &\equiv& 0 \pmod{p}.
\end{eqnarray*} 
Taking the difference of these equations, we observe that
$$
(i^2 - j^2)(x_t - x_{t'}) + 2(i-j)(y_t - y_{t'}) \equiv 0 \pmod{p}.
$$
Since $k\le n^{1/2}$ we have $s_3=(nk)^{1/3}\ge k$ and so $p\ge 4k$.
Therefore $0 < j - i < p$ and hence $i-j \not\equiv 0 \pmod{p}$.  Thus,
$$
(i+j)(x_t - x_{t'}) + 2(y_t - y_{t'}) \equiv 0 \pmod{p}.
$$
In fact, since\footnote{It is here that we needed the tighter
upper bound on the $x$
coordinate values.}
$$|(i+j)(x_t - x_{t'}) + 2(y_t - y_{t'})| \le 2ks_1+2s_2\le 2k(n/k^2)^{1/3} + 2(nk)^{1/3} < p,$$
the only possible solution
to the last equation is when
\begin{equation}
\label{eq:equal}
(i+j)(x_t - x_{t'}) + 2(y_t - y_{t'}) = 0.
\end{equation}
Observe now that if $x_t=x_{t'}$ then from (\ref{eq:equal}) we derive that
$y_t=y_{t'}$ and,
since $h_{3,i}(a_t)=h_{3,i}(a_{t'})$ and $p> 4s_3$ we can
conclude that this would imply
that $z_t=z_{t'}$ which violates our assumption that $a_t\ne a_{t'}$. 
Therefore $x_t\ne x_{t'}$.   
Assume without loss of generality that $x_t-x_{t'} >0$.

Using the fact that $i < j$, by replacing $i$ and $j$ in (\ref{eq:equal}) we see
that
\begin{eqnarray*}
2i(x_t - x_{t'}) + 2(y_t - y_{t'}) < 0 
&\Rightarrow& ix_t + y_t < ix_{t'} + y_{t'} \quad \text{and} \\
2j(x_t - x_{t'}) + 2(y_t - y_{t'}) > 0 
&\Rightarrow& jx_t + y_t > j x_{t'} + y_{t'}.
\end{eqnarray*}
By the lexicographic ordering on triples this means that $a_t <_i a_{t'}$ and
$a_{t'} <_j a_t$, so $a_t$ and
$a_{t'}$ cannot be elements of a common subsequence of $\pi_i$ and $\pi_j$,
and we have arrived at the desired contradiction.
\end{proof}

Lemma~\ref{thm:lcs} thus proves Theorem~\ref{thm:main} for the case that
$s_1,s_2,$ and $s_3$ are integers. 
For the general case, let $s'_1=\left\lceil{(n/k^2)^{1/3}}\right\rceil$ and let
$n'=(s'_1)^3 k^2\le 8n$.
Then $s'_1=(n'/k^2)^{1/3}$ and $s'_2=s'_3=(n'k)^{1/3}=s'_1 k$ are also integers.
Repeating the same argument, we get
$$f_k(n)\le f_k(n')\le 16(n'k)^{1/3}\le 32(nk)^{1/3}.$$
This proves Theorem~\ref{thm:main}.

\bibliographystyle{plain}
\bibliography{Permutations}{}

\appendix

\section{Probabilistic upper bound on $f_k(n)$}
\label{sec:probabilistic}

The \Erdos-Szekeres Theorem stimulated a long line of research into the distribution of the length of the longest increasing subsequence in a \emph{random} permutation~\cite{Ste95}.  The results from this line of research yield a tidy probabilistic argument establishing the upper bound of $f_k(n)$ in (\ref{eqn:upperbound}).

\begin{proposition}
\label{thm:probabilistic}
For any $k < e^{e\sqrt{n}}$, 
$
f_k(n) < 2e\sqrt{n}.
$
\end{proposition}

\begin{proof}
Consider a set $S$ formed by choosing $k$ permutations uniformly at random
from all permutations on $[n]$.  As Frieze showed~\cite[Lemma 1]{Fri91}, the distribution of the length
$L_n$ of the longest increasing subsequence in a random permutation on $[n]$ satisfies
$$
\Pr\left[ L_n \ge 2e\sqrt{n}\right] < e^{-2e\sqrt{n}}.
$$
The length of the longest common subsequence in two random permutations
$\pi_i$ and $\pi_j$ follows the same distribution as the length of the longest
increasing subsequence in the random permutation $\pi_i \circ \pi_j^{-1}$,
so the probability that the set $S$ contains a common subsequence of length at
least $2e\sqrt{n}$ is at most
$
\binom{k}{2} e^{-2e\sqrt{n}} < 1.
$
Therefore, there must exist a set $S$ of $k$ permutations with a longest common
subsequence of length less than $2e \sqrt{n}$.
\end{proof}

\section{Lower bound on $f_k(n)$}
\label{sec:lowerbound}

Beame and Huynh-Ngoc~\cite{BH08} showed that $f_3(n) = n^{1/3}$, which in turn implies that $f_k(n) \ge n^{1/3}$ for every $k \ge 3$.  For completeness, we include the proof of that result here, along with a small improvement for larger values of $k$ obtained with the Pigeonhole Principle.

\begin{proposition}
Let $k\ge 3$ and let $m=m(k)$ be the largest integer such that
$m!<k$ and $m\le n$. Then
$$
f_k(n) \ge \max \left( n^{1/3}, m \right).
$$
\end{proposition}

\begin{proof}
We begin by showing that $f_k(n)\ge f_3(n) \ge n^{1/3}$, using an 
extension of the counting argument in Seidenberg's proof~\cite{Sei59} of
the \Erdos-Szekeres Theorem.

Assume for contradiction that there is a set $S=\{\pi_1,\pi_2,\pi_3\}$ of
permutations on $[n]$ for
which every common subsequence has length strictly less than $n^{1/3}$. 
For every $\ell = 1,\ldots,n$, define
$\phi(\ell) = \big( \phi_{1,2}(\ell), \phi_{1,3}(\ell), \phi_{2,3}(\ell) \big)$,
where $\phi_{i,j}(\ell)$ is the length of the longest common subsequence of
$\pi_i$ and $\pi_j$ that begins with $\ell$.   By assumption,
$\phi_{i,j}(\ell)<n^{1/3}$.  Hence there are strictly fewer than $n$ possible
values of $\phi(\ell)$, which means that there exist $\ell \neq \ell'$ such that
$\phi(\ell) = \phi(\ell')$.  But there must be two permutations $\pi_i$, $\pi_j$
that order $\ell$ and $\ell'$ in the same way, say w.l.o.g. $\ell$ is ordered before
$\ell'$.  This implies that $\phi_{i,j}(\ell) > \phi_{i,j}(\ell')$ so
$\phi(\ell) \neq \phi(\ell')$, a contradiction.

The second part of the theorem follows easily from the Pigeonhole Principle: 
every set $S = \{\pi_1,\ldots,\pi_k\}$
must contain two permutations $\pi_i$ and $\pi_j$ that order
the elements $1,\ldots,m$ identically.  This
completes the proof.
\end{proof}

\end{document}